\documentclass[12pt]{article}
\usepackage[]{times}
\usepackage{amsmath}
\usepackage{paralist}
\usepackage{fullpage}
\usepackage{amssymb}
\usepackage{mathrsfs}
\usepackage{bm}
\usepackage{graphicx}
\usepackage{graphicx}
\usepackage{nameref}
\usepackage{wrapfig}
\usepackage{epsfig}
\usepackage{url}
\usepackage{fullpage}
\usepackage{psfrag}
\usepackage{mathtools}
\usepackage{verbatim}
\usepackage{amsthm}
\usepackage[]{times}
\usepackage{amsmath}
\usepackage{paralist}
\usepackage{aliascnt}
\usepackage{fullpage}
\usepackage{amssymb}
\usepackage{amsrefs}
\usepackage[shortlabels]{enumitem}
\usepackage[colorlinks=true, pdfstartview=FitV, linkcolor=blue,
            citecolor=blue, urlcolor=blue]{hyperref}
\usepackage[usenames]{color}
\definecolor{Red}{rgb}{0.7,0,0.1}
\definecolor{Green}{rgb}{0.0,0.45,0.0}
\usepackage[sans]{dsfont}

\usepackage{mathrsfs}
\usepackage{bm}
\usepackage{graphicx}
\usepackage{graphicx}
\usepackage{wrapfig}
\usepackage{epsfig}
\usepackage{url}
\usepackage{fullpage}
\usepackage{psfrag}
\usepackage{verbatim}

\newcommand{\cV}{\mathbb{V}}
\newcommand{\cT}{\mathcal{T}}
\newcommand{\IND}[1] {{ \mathds{1}_{ #1 }} }

\newcommand{\bX}{\mathbf{X}}

\newcommand{\rme}{\mathrm{e}}

\def\NewTheorem#1#2{%
	\newaliascnt{#1}{thmm}
	\newtheorem{#1}[#1]{#2}
	\aliascntresetthe{#1}
	\expandafter\def\csname #1autorefname\endcsname{#2}
}

\numberwithin{equation}{section}
\NewTheorem{thm}{Theorem}
\NewTheorem{lemma}{Lemma}
\NewTheorem{cor}{Corollary}
\NewTheorem{prop}{Proposition}

\theoremstyle{remark}
\NewTheorem{Note}{Note}
\NewTheorem{remark}{Remark}

\theoremstyle{definition}
\NewTheorem{definition}{Definition}
\newtheorem*{problem}{Explosion problem}
\NewTheorem{example}{Example}

\begin{document}

\title{Doubly Stochastic Yule Cascades (Part I): The explosion problem in the
time-reversible case}
\author{Radu Dascaliuc\thanks{Department of Mathematics,  Oregon State University, Corvallis, OR, 97331. {dascalir@math.oregonstate.edu}}
\and Tuan N.\ Pham\thanks{Department of Mathematics,  Brigham Young University, Provo, UT, 84602. {tuan.pham@mathematics.byu.edu}}
\and Enrique Thomann\thanks{Department of Mathematics,  Oregon State University, Corvallis, OR, 97331. {thomann@math.oregonstate.edu}}
\and
Edward C.\ Waymire\thanks{Department of Mathematics,  Oregon State University, Corvallis, OR, 97331. {waymire@math.oregonstate.edu}}
}

\maketitle

\begin{abstract}
Motivated by the probabilistic methods for nonlinear differential equations introduced by McKean (1975) for the Kolmogorov-Petrovski-Piskunov (KPP) equation, and by Le Jan and Sznitman (1997)  for the incompressible Navier-Stokes equations (NSE),  we 
identify a new class of stochastic cascade models, referred to as 
{\it doubly stochastic Yule cascades}. We establish non-explosion
 criteria under the assumption that the randomization of Yule intensities from
 generation to generation is by an ergodic time-reversible Markov process. 
 In addition to the cascade models that arise in the analysis of certain
 deterministic nonlinear differential equations, this model includes the
multiplicative branching random walks, the branching Markov processes, 
  and the stochastic generalizations of the percolation and/or cell ageing
  models  introduced by Aldous and Shields (1988) 
 and independently by Athreya (1985). 
 \end{abstract}

\section {Background Motivation and Definition of Doubly Stochastic Yule Cascades}

Doubly stochastic Yule cascades represent a new class of models that involve a branching structure governed by exponential waiting times with random intensities. 
This class of models is quite diverse from the perspective of nonlinear PDEs to purely probabilistic models of stochastic phenomena, such as percolation and aging models. Our particular motivation comes from a class of 
evolutionary PDEs which, after suitable normalization in the Fourier space, can be expressed in a mild-type form:
\begin{equation}\label{mild_PDE_general}
u(t,\xi)=u_0(\xi)\rme^{-\lambda(\xi)\,t}+\int_{0}^{t}\lambda(\xi)\rme^{-\lambda(\xi)\,s}\int_{\mathbb{R}^d} B\left(u(t-s,\eta), u(t-s,\xi-\eta)\right) H(\eta|\xi)\, d\eta\,d s,
\end{equation}
where $u_0$ represents the initial data, $\lambda(\cdot)$ represents linear part of the PDE (a Fourier multiplier), $B(\cdot,\cdot)$ represents a nonlinearity of quadratic type, and $H(\cdot|\xi)$ is a $\xi$-dependent probability kernel. 

Two particular examples of such PDEs that we consider are the incompressible 3D Navier-Stokes equations (NSE) and Kolmogorov-Petrovski-Piskunov equation (KPP), also known as Fisher-KPP equation (see \autoref{examples}). A remarkable observation, dating back to McKean's original work on KPP \cites{mckean1975application,MB_1983} and the Le Jan and Sznitman's paper \cite{YLJ_AS_1997} for NSE, is that such a mild formulation can be interpreted as an expected value of a stochastic process $\bX(\xi,t)$ built, via the quadratic term $B(\cdot,\cdot)$, on a binary tree structure governed by exponential waiting times between branchings. The exponential intensities $\lambda(\cdot)$ are in turn random and governed by the distribution $H(\cdot|\cdot)$. Thus, the problems from the analysis of the PDE  (\ref{mild_PDE_general}) can be re-cast in terms of properties of the ``solution" stochastic process $\bX$.  In particular, a basic question about the branching structure becomes that of {\em stochastic explosion}: does the stochastic cascade generate infinitely many branches by a finite time $t>0$? An answer to this question directly affects existence and uniqueness properties of solutions to (\ref{mild_PDE_general}).
For example, the classical branching Brownian motion associated with the classical KPP equation is non-explosive, resulting in uniqueness for solutions of the corresponding initial value problem \cite{mckean1975application}, while the branching diffusion processes associated with certain generalizations to the KPP equation are explosive, leading to both non-uniqueness and finite-time blowup of solutions \cite{ito2012diffusion}*{pp. 206-211} (see also \cites{Fujita, Lopez, Nagasawa}). In the NSE case, Le Jan and Sznitman \cite{YLJ_AS_1997} circumvented the problem of stochastic explosion by using a thinning procedure. However, the thinning masks possible explosion of the underlying stochastic structures, which could hint at possible lack of well-posedness of NSE in certain settings (see e.g.\  \cites{RD_NM_ET_EW_2015, RD_ET_EW_2019, DTT_22}). 
For a background on stochastic cascades arising from PDEs and the role of stochastic non-explosion in the existence and uniqueness theory of the solutions, see \autoref{PDE-casc}.

\begin{remark}
It is worth noting that the explosion in a stochastic cascade corresponding to a PDE is not equivalent to finite-time blowup of the solutions. Rather, it is directly connected to the existence and uniqueness of the stochastic processes whose expectation yields a solution to \eqref{mild_PDE_general} (see \autoref{PDE-casc}). In fact, there are simple equations associated with non-explosive cascades and admitting finite-time blowup solutions \cite{alphariccati, DTT_22}.
\end{remark}

 There are differences between the classical branching Brownian diffusion structures
associated with KPP-type equations in the physical space and the branching random walk structures associated with NSE in the Fourier space.  First, the former are scalar (concentration) equations, while the latter are 
vectorial (velocity) equations. Secondly, and more importantly, the successive generation of particles along each
tree path is not independent, although there is Markov dependence. The lack of independence complicates the problem of explosion/non-explosion. To our knowledge, 
a basic theory of non-explosion for such Markov-dependent branching structures is unavailable in the published literature.  Therefore,
the purpose of this paper is twofold: first, to identify a general stochastic structure which is flexible enough to accommodate a variety of similar models, and second, to determine general criteria for {non-explosion}.  As the starting point, 
let us recall the classical Yule cascade.

On the full infinite binary tree  ${\mathbb T}=\{\theta\}\cup(\cup_{n=1}^\infty\{1,2\}^n)$, let us denote by $\theta$ the root. For a path $s\in\partial{\mathbb T} =
\{1,2\}^\infty$, we
denote by $s|n = (s_1,\dots,s_n)$, where $n\ge 1$, the restriction of $s$ to the first $n$
generations, with the convention that $s|0 = \theta$.  The generational height of a vertex $v=s|n$ is denoted by $|v|=n$.
A vertex uniquely determines the genealogical sequence between
it and the root.

As a counting process, the classical Yule cascade is typically introduced as a continuous parameter 
Galton-Watson branching process with  single progenitor, with
offspring distribution $p_2=1$, and with infinitesimal rate parameter $\lambda >0$ (or equivalently, as a pure birth Markov process with rate $\lambda >0$). 
The case $\lambda=1$ is referred to as the {\em standard Yule cascade} and be viewed as a tree-indexed family $\{T_v\}_{v\in\mathbb T}$ of i.i.d.\  mean-one exponential random variables. Correspondingly, the classical Yule cascade with the intensity parameter $\lambda$ becomes the family $\{\lambda^{-1}T_v\}_{v\in\mathbb T}$, which is a re-scaling of the standard Yule cascade.

Viewed this way, the
above counting process can be defined by the cardinalities $N(t) = \# V(t), t\ge 0$,
of the set-valued evolution
\begin{equation}
\label{yuledef}
V(t) = \begin{cases} \{\theta\} &\mbox{if } t\le {\frac{1} {\lambda}}T_\theta,\\
\left\{v\in{\mathbb T}: \sum_{j=0}^{|v|-1}{\frac{1}{\lambda}}T_{v|j}<t\le  \sum_{j=0}^{|v|}{\frac{1}{\lambda}}T_{v|j}\right\}
&\mbox{otherwise}.
\end{cases}
\end{equation}

More generally, one can define a \emph{non-homogeneous Yule cascade} with positive parameters (intensities) $\{\lambda_v\}_{v\in\mathbb{T}}$ as a tree-indexed family $\{\lambda_v^{-1}T_v\}_{v\in\mathbb{T}}$ where $\{T_v\}_{v\in\mathbb{T}}$ is the standard Yule cascade.

As in the case of doubly stochastic Poisson process, one may allow the intensities of a non-homogeneous Yule cascade to be positive random variables.  This essentially defines the {\it doubly stochastic Yule cascade.}   

\begin{definition}
\label{DSY_def}
We refer to a tree-indexed family
of random variables $\{{\lambda_v^{-1}{T_v}}\}_{v\in\mathbb T}$,
where $\{\lambda_v\}_{v\in\mathbb{T}}$ is a tree-indexed family of positive random variables independent of the standard Yule cascade $\{T_v\}_{v\in\mathbb{T}}$, as a {\it doubly stochastic Yule \textup{(}DSY\textup{)} cascade} with intensities $\{\lambda_v\}_{v\in\mathbb{T}}$.
\end{definition}

\begin{remark}Equivalently to \autoref{DSY_def}, the DSY cascade can viewed
as a pair of tree-indexed families of positive random variables
$\Lambda = \{\lambda_v\}_{v\in{\mathbb T}}$ and $\{T_v\}_{v\in{\mathbb T}}$
such that conditionally given $\Lambda$,  $\{{\lambda_v}^{-1}{T_v}\}_{v\in{\mathbb T}}$ is 
distributed as a {\em non-homogeneous} Yule cascade with corresponding set of parameters 
$\Lambda$.  With this definition, it is relatively straightforward that $\{T_v\}_{v\in{\mathbb T}}$ 
must be a standard Yule cascade, independent of $\Lambda$.
\end{remark}

Motivated by the dynamical systems nature of (\ref{mild_PDE_general}), we consider an {\em evolutionary process} associated to DSY, a straightforward generalization of (\ref{yuledef}): 
\begin{equation}
\label{DSY_ev_def}
V(t) = \begin{dcases} \{\theta\} &\mbox{if } t\le {\frac{1} {\lambda_\theta}}T_\theta,\\
\left\{v\in{\mathbb T}: \sum_{j=0}^{|v|-1}{\frac{1}{\lambda_{v|j}}}T_{v|j}<t\le  \sum_{j=0}^{|v|}{\frac{1}{\lambda_{v|j}}}T_{v|j}\right\}
&\mbox{otherwise}.
\end{dcases}
\end{equation}
One can interpret $V(t)$ as the set of vertices of the DSY cascade that cross time $t>0$ (\autoref{cascade}).

\begin{figure}[h]
\centering
\includegraphics[scale=1]{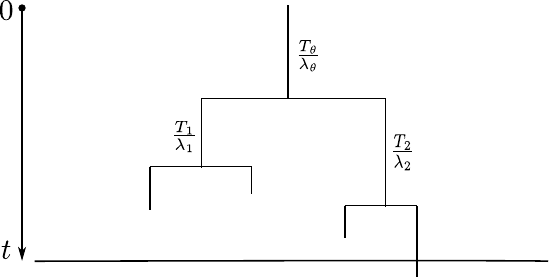}
\caption{Doubly stochastic Yule  cascade with random intensities $\{\lambda_v\}$.}
\label{cascade}
\end{figure}

A basic probability problem associated with the stochastic evolution of (\ref{DSY_ev_def}) is the {\em explosion problem}. The paper \cite{IB_YP94} is something of a 
loosely related precursor to the question 
of interest here for the DSY cascade.
 That is:
\begin{problem}
\emph{Will the cascade reach every finite time horizon $t>0$ in finitely many branchings \textup{(}{non-explosion}\textup{)}, or can it happen that there will be infinitely many branches before a finite time horizon \textup{(}{explosion}\textup{)}? See \autoref{cascade} for a visual representation of the problem.}
\end{problem}

The explosion problem can be formulated using the notion of \emph{explosion time} as follows.
\begin{definition}
\label{expl_time}
The {\em explosion time} of a DSY cascade $\{\lambda_v^{-1}T_v\}_{v\in\mathbb{T}}$ is a $[0,\infty]$-valued random variable $\zeta$ defined by
\[\zeta =\underset{n\ge 0 }{\mathop{\sup }}\,\underset{|v|=n}{\mathop{\min }}\,\sum\limits_{j=0}^{n}{\frac{{{T}_{v|j}}}{\lambda_{v|j}}}.\]
The event of {\em explosion} and {\em non-explosion} is defined by $[\zeta<\infty]$ and $[\zeta=\infty]$, respectively. The cascade is said to be {\em non-explosive} if $\mathbb{P}(\zeta=\infty)=1$, and {\em explosive} if $\mathbb{P}(\zeta=\infty)<1$.
\end{definition}
\begin{remark}\label{infinitepath} 
Intuitively, the explosion time of a DSY cascade is the shortest path. Specifically, for each sample point $\omega$ there exists a path $s=s(\omega)\in\partial\mathbb{T}$ such that $\zeta(\omega)=\sum_{j=0}^\infty\frac{T_{s|j}(\omega)}{\lambda_{s|j}(\omega)}$. To see this, starting at the root $\theta$, this path can be constructed recursively thanks to the ``inherited'' structure of the explosion time. 
Namely, we go to the left branch if the left subtree has a smaller explosion time than the right subtree. Otherwise, we go to the right branch. The notion of explosion is consistent with the intuitive idea illustrated in \autoref{cascade}: on the event of explosion, there exists a random path that never reaches some finite time $t$, and thus the tree has generated infinitely many vertices by that time.
\end{remark}
\begin{remark}
Note that $\lim_{t\to\zeta^-}|V(t)|=\infty$.
\end{remark}
While it is well-known that the standard Yule cascade is non-explosive  \cite[p.\ 450]{WF1968}, the present paper focuses on the explosion
problem for doubly stochastic Yule cascades.

As already noted, DSY cascades arise naturally in the 
analysis of stochastic cascade models of nonlinear differential equations such as the Navier-Stokes equation, the KPP equation, as well as the complex Burgers equation \cite{RD_NM_ET_EW2019}, and the $\alpha$-Riccati equation \cite{RD_ET_EW_2019}.  
This framework may also be viewed as a 
doubly stochastic
generalization of a class
of random cascade models introduced in \cite{KA_1985}, and
independently in \cite{DA_PS_1988}, in which the times between
branchings at the $|v|$-th generation are (deterministically) scaled to be
 exponentially distributed with intensities $\lambda_v = \alpha^{-|v|}$, for a positive
 parameter $\alpha$.   This deterministically changing  
rate of splitting according to generation is analyzed in the case
$0<\alpha\le 1$ in \cite{KA_1985} and
\cite{RD_NM_ET_EW2017},  and in the case $\alpha> 1$ in \cite{DA_PS_1988}. In the latter 
reference, the model is interpreted in terms of both data compression 
and percolation.  Recently, such models have also been considered
for  important cellular biology questions related to ageing and cancer,
 where 
 generational dependent cell division rates occur
 and decrease
with generations; see e.g. \cites{KB_PP_2010, CL_RDP_BFS_2012}
and other related references in the medical and biological literature.

\begin{remark} For the percolation model \cite{DA_PS_1988}, the event of explosion corresponds
 to the occurrence of a cluster of infinitely many \lq\lq wet sites\rq\rq connected
 to the root in finite time.   For the biological model \cite{KB_PP_2010}, the
 ageing  is represented by non-explosive conditions for the cascade.
 \end{remark}

From the point
of view of differential equations, these models also
correspond to a class of $\alpha$-Riccati differential equations 
analyzed in \cite{KA_1985} and \cite{RD_ET_EW_2019}.

The DSY cascades introduced in \autoref{DSY_def} are quite general, and in order to consider the explosion problem we will further assume certain Markov-chain structure underlying the random intensities $\lambda_v$ (see \autoref{DSYMcascade}), with transition probabilities satisfying time-reversibility constraints. We note that in the non-homogeneous case ($\lambda_v$'s are constant) various approaches, such as the martingale or semigroup techniques (discussed in \autoref{settings}) can be taken to study explosion problems. In the case of random intensities $\lambda_v$, which is required by our applications, the standard available tools are limited, even in the case of Markov transitions for the intensities along a path. This  necessitates a new approach.

Our main result is a general non-explosion criterion inspired by large-deviation techniques and expressed in terms of a bound on a spectral radius of an associated linear operator (see \autoref{noexplosion}). 
This theorem and its corollaries are sufficient to determine 
non-explosion in a variety of interesting DSY cascades, such as those associated with NSE, KPP, and certain stochastic models. In particular, following  Orum \cite[Sec.\ 7.9]{orum},
 our approach to KPP identifies a new DSY cascade structure that can be naturally associated with KPP in {\em Fourier space}, which is quite different from the branching motion associated with KPP in physical space settings. Although our interest is mainly on DSY cascades on a binary tree, which are well-suited with PDEs with quadratic nonlinearity, our techniques can be applied to tree structures with random number of offspring  (see \autoref{non-binary}, \autoref{nonexpbound-gw}). 

While we focus the present paper on the time-reversible case,
the problem is of interest for non-reversible, in fact
  non-ergodic, cases as well; see \cite{DTTW_22} for explosion criteria by
   methods that do not require reversibility.

The paper is organized as follows.  In \autoref{settings}, we define a specific type of DSY cascades to consider the non-explosion problem. We then formulate and prove the main results regarding non-explosion in \autoref{main}. An extension of the main results to non-binary trees is discussed in \autoref{non-binary}. In \autoref{examples}, we apply our non-explosion criteria to the classical birth and death processes and 
to stochastic cascades associated with NSE and KPP equations. 
We finish with some concluding remarks in \autoref{conclusion}. Some background about the connection between the explosion/non-explosion problems and the well-posedness problems of evolutionary PDEs is provided in \autoref{PDE-casc}.

\section{Type $(\mathcal{M})$ Doubly Stochastic Yule Cascade}\label{settings}

In order to analyze the explosion problem for DSY cascades, 
we need additional assumptions on the intensities $\lambda_{v|j}$ in \autoref{DSY_def}. Again, we are motivated by the DSY cascade that underlies equation (\ref{mild_PDE_general}). For the purpose of illustration, one may view (\ref{mild_PDE_general}) as the mild formulation of the 3-dimensional NSE in the Fourier space (see \autoref{PDE-casc}). At each wave vector $\xi\in\mathbb{R}^3$, a DSY cascade is generated with $\lambda_{v} = \lambda(W_{v})>0$ where $W_\theta\equiv\xi$ and
$W_v$, for $v\neq\theta$, is random wave vector distributed according to a probability kernel $H$, consistent with the governing equations. Although the wave vectors $W_v$ 
are vectors in $\mathbb{R}^3$, the explosion time $\zeta(\xi)$ for the tree-indexed random field depends only on the intensities $\lambda_v$, which in turn depends only  magnitudes $|W_v|$. This family constitutes a branching Markov process on a scalar state space.

In the typical cases, such as NSE or KPP,  the transition probability kernel $H$ is such that  the family $\{X_v=W_v\}_{v\in\mathbb{T}}$ is a binary branching
Markov process on $\mathbb{R}^d$. More generally, Markov structure is a natural extension of independence in stochastic models, which motivates the following definition.

\begin{definition}
\label{DSYMcascade}
We say that a DSY cascade $\{\lambda_v^{-1}T_v\}_{v\in\mathbb{T}}$ is of type ($\mathcal{M}$) if $\lambda_v=\lambda(X_v)$, where $\lambda$ is a $(0,\infty)$-valued function and $\{X_v\}_{v\in\mathbb{T}}$ is a tree-indexed family of random variables satisfying:
\begin{enumerate}[(A)]
\item For any path $s \in\partial{\mathbb T}$,
the sequence $X_{s|0}, X_{s|1}, X_{s|2},\dots$ is a time-homogeneous Markov chain on a measurable state space $(S,\mathcal{S})$.
\item For any path $s\in\partial{\mathbb T}$, the transition probability of the Markov chain $X_{s|0}, X_{s|1}, X_{s|2},\dots$ does not depend on $s$.
\end{enumerate}
\end{definition}

Our main goal is to provide criteria for non-explosion of the type ($\mathcal{M}$) DSY cascades (i.e.\ $\zeta=\infty$ a.s.\ as defined in \autoref{expl_time}).

To place the explosion problem in the perspective of Markov semigroups, we close this section by considering a particular case where $\lambda_v$ are deterministic, i.e.\ the non-homogeneous Yule cascades. Let $\mathcal{E}$ be the family of all finite sets $W\subset \mathbb{T}$ such that either $W=\{\theta\}$ or $\{\theta\}\neq W= V^v$ for some $V\in{\mathcal E}$ and $v\in V$, where $V^v=V\backslash\{v\}\cup\{v*1,v*2\}$. Here $v*1$ and $v*2$ denote the two offspring of vertex $v$. Endow $\mathcal{E}$ with the discrete topology 
defined by the usual discrete metric (i.e. $d(V,V)=0$ and $d(V,W)=1$ for $V,W\in\mathcal{E}$, $V\ne W$). 
The non-explosive non-homogeneous Yule cascades with intensities $\{\lambda_v\}_{v\in\mathbb{T}}$ admit a semigroup formulation in which the set-valued evolution (\ref{DSY_ev_def}) can be represented as 
a semi-group $\{S_t:t\ge 0\}$ of positive linear contraction operators on 
 the space $C_0({\mathcal E})$ of continuous functions  
on ${\mathcal E}$ vanishing at infinity.

In particular, the infinitesimal
transition rates are given by 
\begin{equation}q(V,W) = \begin{cases*}
      \lambda_v & if $W = V^v$\ for some $v\in V$, \\
      -\sum_{v\in V}\lambda_v & if $W=V$,\\        
        0    & otherwise.
    \end{cases*}
  \end{equation}
  If $0<\lambda_v\le 2^{-|v|}$ for all $v\in{\mathbb T}$, then 
  for all $V\in\mathcal{E}$, the rates
 $|q(V,V)| = \sum_{v\in V}\lambda_v
 \le \sum_{v\in V}2^{-|v|}= 1$ are bounded (see \cite{RD_NM_ET_EW2017}). In this case, 
 $S_t = e^{tL}$ is the uniquely
  associated strongly continuous semigroup for these rates where
  \begin{equation}
\label{infcore}
Lf(V) = \sum_{W\in{\mathcal E}}q(V,W)(f(W)-f(V)) = 
\sum_{v\in V}\lambda_{v}(f(V^v)-f(V)), \ V\in{\mathcal E},\ 
f\in C_{0}({\mathcal E}).
\end{equation}
The non-explosion problem may be viewed as conditions on the rates
for which $(L,{\mathcal D})$ continues to generate a {\it conservative} 
positive contraction semigroup, i.e., 
$\sup_{0\le f\le 1}S_tf(V) = 1$ for all $V\in{\mathcal E}, t\ge 0$, on 
the state space ${\mathcal E}$, or for the existence of unique global
solutions to the
Cauchy problem 
\begin{equation}
\frac{\partial u}{\partial t} = Lu, \quad u(0) = u_0\in {\mathcal D}\subset C_0({\mathcal E}),
\end{equation} 
where $u(t,V) = S_tu_0(V), \ V\in{\mathcal E}, \ t\ge 0$.
On the other hand, 
explosion leads to `compactifications' of the state space ${\mathcal E}$
and non-uniqueness of transition semigroups, also of interest.
One may note that ${\mathcal E}$ also embodies a tree ancestory partial order. 
In any case, from this perspective the DSY cascades may be viewed as 
(semi-Markov) non-homogeneous Yule evolutions in a random environment.
The approach we adopt for the explosion problem in this paper is related
in so far as the formulation is in terms of transition operators for a related
discrete parameter process, rather than directly with the continuous parameter
process $V(t), t\ge 0$. While Lyapounov techniques
could be fruitful for non-explosion criteria for the non-homogeneous 
Yule cascade, e.g., see \cite{MH_JM2006}, necessary and sufficient
 explosion criteria for these have recently been obtained by
methods of the present paper by \cite{TP_2021}.

In the general framework of a type ($\mathcal{M}$) DSY cascade, the Markov operator $L$ is itself random, which makes its analysis challenging.

\section{Main Results}\label{main}
The following key lemma identifies the nature of the problem
as a competition between the branching rate and the behavior of the 
intensities along paths.
\begin{lemma}[Key Lemma]
\label{nonexpbound}
Let $\{\lambda_v^{-1}T_v\}_{v\in\mathbb{T}}$ be a DSY cascade such that for each $s\in\partial\mathbb{T}$, the distribution of the sequence $\lambda_{s|0}$, $\lambda_{s|1}$, $\lambda_{s|2}$, \ldots does not depend on $s$. Then for $a>0$ and an arbitrary fixed path ${s}\in\partial{\mathbb T}$,
\begin{equation}
\label{noexpbound}
{\mathbb E} e^{-a\zeta}\le \liminf_{n\to\infty}2^n{\mathbb E}\prod_{j=0}^n{\frac{\lambda_{{s}|j}}{
a+\lambda_{{s}|j}}},
\end{equation}
where $\zeta$ is given in Definition \ref{expl_time}.
Consequently, if 
\begin{equation}\label{noexpbound1}
\liminf_{n\to\infty}2^n{\mathbb E}\prod_{j=0}^n{\frac{\lambda_{{s}|j}}{
a+\lambda_{{s}|j}}}=0
\end{equation} 
for some $a>0$ then the cascade is non-explosive. 
\end{lemma}
\begin{proof}

By Fatou's lemma, some large deviation estimates \cite{JB_1977}
 and the simple bound on a maximum by the sum
 \begin{eqnarray*}
\label{maxbysum}
{\mathbb E} e^{-a\zeta} &\le& \liminf_{n\to\infty}{\mathbb E}e^{-\min_{|v|=n}\sum_{j=0}^na \lambda_{v|j}^{-1}T_{v|j}}\nonumber\\
&\le& \liminf_{n\to\infty}{\mathbb E}\sum_{|v|=n}e^{-\sum_{j=0}^na \lambda_{v|j}^{-1}T_{v|j}}\nonumber\\
&=& \liminf_{n\to\infty}{\mathbb E}2^ne^{-\sum_{j=0}^na \lambda_{{s}|j}^{-1}T_{{s}|j}}\nonumber\\
&=& \liminf_{n\to\infty}2^n{\mathbb E}\prod_{j=0}^n{\frac{\lambda_{{s}|j}}{
a+\lambda_{{s}|j}}}
\end{eqnarray*}
where ${s}\in\partial\mathbb{T}$ is an arbitrary fixed path. If the right hand side of \eqref{noexpbound} is equal to zero for some number $a>0$, then $\mathbb{E}e^{-a\zeta}=0$. This leads to $\zeta=\infty$ a.s.
\end{proof}

\begin{remark}
Similar bounds are routine in the computation
of extremal particle speeds for branching random walks having
i.i.d.\ displacements; see e.g.\ \cite{ZS_2015}. However, there appears to be little literature on 
the general theory of branching random walks for more general
ergodic Markov displacements treated here; see
\cite{EW_SW_1997} for another example.  
We are unaware of a theory to determine the speed of the left-most
particle for such  Markov dependent branching random walks.  Similar remarks
apply to first passage percolation, e.g. \cite{AA_MD_JH2017, IB_YP1994}.
\end{remark}

The main results in this paper give sufficient conditions for \eqref{noexpbound1} to hold for DSY cascades of type ($\mathcal{M}$) under the assumption that along each path $s\in\partial\mathbb{T}$ the Markov chain 
$X_{s|0}$, $X_{s|1}$, $X_{s|2}$,\ldots is time reversible.
Let $\gamma$ is an invariant probability distribution of the Markov process on the state space $(S,\mathcal{S})$. For each $a\ge 0$, one can define a 
positive contraction
operator $T_a: L^2(\gamma)\to L^2(\gamma)$ by 
\begin{equation}\label{Ta}T_af(x) = 
\frac{\lambda(x)}{a+\lambda(x)}\int_Sf(y)p(x,dy),
\end{equation}
where $p(x,dy)$ is the one-step transition probability of the Markov process.
In particular, 
$$T_0 f (x)= \mathbb{E}_x [f(X_1)] = \int_S f(y)p(x,dy). 
$$
Note that $T_af(x) = 
g_a(x)T_0f(x)$, where
\begin{equation}\label{ga}
g_a(x) = \frac{\lambda(x)}{a+\lambda(x)}.
\end{equation}
The time reversibility property of the Markov chain 
makes $T_0$  a self-adjoint operator on $L^2(\gamma)$, i.e.
$$
\langle f_1, T_0 f_2\rangle_\gamma = \langle T_0f_1, f_2\rangle_\gamma\ \ \ \forall\, f_1, f_2 \in L^2(\gamma).
$$

The main theorem to be proven is the following.

\begin{thm}
\label{noexplosion}
Let $\{\lambda(X_v)^{-1}T_v\}_{v\in\mathbb{T}}$ be a DSY cascade of type ($\mathcal{M}$) such that along each path $s\in\partial\mathbb{T}$ the Markov process $X_{s|0}$, $X_{s|1}$, $X_{s|2}$,\ldots is time-reversible with respect to an invariant probability measure $\gamma$. Suppose that for some $a >0$,
\begin{equation}\label{eq:23211}\limsup_{n\to\infty}\sqrt[n]{\langle 1,T_a^n1\rangle_\gamma} <\frac{1}{2}.\end{equation}
Then 
\begin{enumerate}[(a)]
\item for $\gamma$-a.e. $x\in S$, the cascade is non-explosive for initial state $X_0=x$. 
\item If, in addition, $p(x_0,dy)\ll\gamma(dy)$ for some $x_0\in S$ then the cascade associated with the initial state $X_\theta=x_0$ is non-explosive. 
\end{enumerate}
\end{thm}
The proof follows from a few preliminary calculations.  For simplicity of exposition, we denote
$X_j = X_{{s}|j}$ for an arbitrary fixed path ${s}\in\partial{\mathbb T}$.

\begin{lemma} For any $a\ge 0$ and $f\in L^2(\gamma)$,
\begin{equation}
\label{eq11191}
{\mathbb E}_x\prod_{j=0}^ng_a(X_j)f(X_{n+1}) 
= T_a^{n+1}f(x),
\end{equation}
where $g_a$ is defined by \eqref{ga}.
\end{lemma}

\begin{proof}
For $n=0$, one has
\begin{equation*}
{\mathbb E}_x g_a(X_0)f(X_{1}) = \frac{\lambda(x)}{a+\lambda(x)}
\int_Sf(y)p(x,dy) 
=T_af(x).
\end{equation*}
For $n\ge 1$,
\begin{eqnarray*}
{\mathbb E}_x\prod_{j=0}^n{g_a(X_j)}f(X_{n+1})
&=& {\mathbb E}_x\prod_{j=0}^{n}{g_a(X_j)}{\mathbb E}[f(X_{n+1})\vert\sigma(X_1,\dots,X_n)]\nonumber\\
&=& {\mathbb E}_x\prod_{j=0}^n{g_a(X_j)}\int_Sf(z)p(X_n,dz)
\nonumber\\
&=&  {\mathbb E}_x\prod_{j=0}^{n-1}{g_a(X_j)}
\int_Sf(y)\frac{\lambda(X_n)}{a+\lambda(X_n)}p(X_n,dy)
\nonumber\\
&=&{\mathbb E}_x\prod_{j=0}^{n-1}{g_a(X_j)}
T_af(X_n).
\end{eqnarray*}
Here $\sigma(X_1,\dots,X_n)$ denotes the $\sigma$-field generated by $X_1,\ldots,X_n$. The result then follows by induction.

\end{proof}

Let us proceed to the proof \autoref{noexplosion} as follows.
\begin{proof}[Proof of \autoref{noexplosion}]
For $f\in L^2(\gamma)$, by integrating \eqref{eq11191} against $\gamma(dx)$ and noting that $T_af(x) = 
g_a(x)T_0f(x)$, one gets
\begin{eqnarray*}
\label{eq11193}
\mathbb{E}_\gamma\prod_{j=0}^n{g_a(X_j)}f(X_{n+1}) 
&=&
\langle 1,T_a^{n+1}f\rangle_\gamma 
=
\langle 1,g_aT_0T_a^nf\rangle_\gamma \\
&=&
\langle g_a,T_0T_a^nf\rangle_\gamma 
=\langle T_0 g_a,T_a^nf\rangle_\gamma.
\end{eqnarray*}
By taking $f=1$, one gets ${{\mathbb{E}}_{\gamma }}\prod_{j=0}^{n}{{{g}_{a}}({{X}_{j}})}\le {{\left\langle 1,T_{a}^{n}1 \right\rangle }_{\gamma }}$. Then
\begin{equation*}
\limsup_{n\to\infty}\frac{1}{n}\log 2^n {\mathbb E}_\gamma \prod_{j=0}^n{g_a(X_j)} \le 
\underset{n\to \infty }{\mathop{\lim \sup }}\,\frac{1}{n}\log \left( {{2}^{n}}{{\left\langle 1,T_{a}^{n}1 \right\rangle }_{\gamma }} \right)<0.
\end{equation*}
This implies that there exists $\delta>0$ such that
$
\log 2^n {\mathbb E}_\gamma \prod_{j=0}^n{g_a(X_j)} 
\le - n \delta$
for all but finitely many $n$. 
Thus, 
\begin{equation}\label{1231201}
2^n{\mathbb E}_\gamma \prod_{j=0}^n\frac{\lambda(X_{{s}|j})}{a+\lambda(X_{{s}|j})} 
\le e^{-n\delta}.
\end{equation}
According the estimate \eqref{noexpbound} and Fatou's Lemma,
\begin{eqnarray}
\nonumber\int_{S}{{{\mathbb{E}}_{x}}{{e}^{-a\zeta }}\gamma (dx)}&\le& \int_{S}{\underset{n\to \infty }{\mathop{\lim \inf }}\,{{2}^{n}}{{\mathbb{E}}_{x}}\prod\limits_{j=0}^{n}{\frac{\lambda ({{X}_{s|j}})}{a+\lambda ({{X}_{s|j}})}}\gamma (dx)}\\
\nonumber&\le& \underset{n\to \infty }{\mathop{\lim \inf }}\,\int_{S}{{{2}^{n}}{{\mathbb{E}}_{x}}\prod\limits_{j=0}^{n}{\frac{\lambda ({{X}_{s|j}})}{a+\lambda ({{X}_{s|j}})}}\gamma (dx)}\\
\nonumber&=&\underset{n\to \infty }{\mathop{\lim \inf }}\,{{2}^{n}}{{\mathbb{E}}_{\gamma }}\prod\limits_{j=0}^{n}{\frac{\lambda ({{X}_{s|j}})}{a+\lambda ({{X}_{s|j}})}}\\
\label{1231202}&=& 0.
\end{eqnarray}
Therefore, $\mathbb{E}_xe^{-a\zeta}=0$ for $\gamma$-a.e. $x\in S$. Consequently, for $\gamma$-a.e. $x\in S$ the cascade associated with initial state $X_\theta=x$ is non-explosive.

Now suppose that $p(x,dy)\ll \gamma(dy)$ for some $x\in S$.  
With $X_\theta=x$, the explosion time can be written as $\zeta =T_\theta\lambda(x)^{-1}+\min \left\{ {{\zeta }^{(1)}},{{\zeta }^{(2)}} \right\}$ where
\[{{\zeta }^{(\sigma )}}=\underset{n\ge 1 }{\mathop{\sup }}\,\underset{|v|=n}{\mathop{\min }}\,\sum\limits_{j=1}^{n}{\frac{{{T}_{\sigma *v|j}}}{\lambda ({{X}_{\sigma *v|j}})}},\ \ \ \ \ \sigma \in \{1,2\}.\]
Here, the notation $\sigma*v$ denotes the vertices on the subtree rooted at $\sigma$. The explosion time is then equal to the holding time at the root $\theta,$ appropriately scaled, plus the smaller of the explosion times of the two subtrees re-rooted at $\sigma= 1,\,2,$ respectively.
Note that $\zeta^{(\sigma)}$ is the explosion time of the DSY cascade $\left\{ \frac{T_{v}^{(\sigma )}}{\lambda (X_{v}^{(\sigma )})}:\,v\in \mathbb{T} \right\}$ where $T_{v}^{(\sigma )}={{T}_{\sigma *v}}$ and $X_{v}^{(\sigma )}={{X}_{\sigma *v}}$. We have
\begin{eqnarray}
\label{1231204}{{\mathbb{E}}_{x}}{{e}^{-a\zeta }}&\le& {{\mathbb{E}}_{x}}[{{e}^{-a\min \{{{\zeta }^{(1)}},{{\zeta }^{(2)}}\}}}]\le \sum_{\sigma=1}^2\mathbb{E}[{{e}^{-a{{\zeta }^{(\sigma)}}}}|{{X}_{\theta }}=x].
\end{eqnarray}
Fix $\sigma\in\{1,2\}$. By conditioning on $X_\sigma$,
\begin{eqnarray}
\nonumber\mathbb{E}[{{e}^{-a{{\zeta }^{(\sigma )}}}}|{{X}_{\theta }}=x]&=&\int_{S}{\mathbb{E}[{{e}^{-a{{\zeta }^{(\sigma )}}}}|{{X}_{\theta }}=x,{{X}_{\sigma }}=y]p(x,dy)}\\
\nonumber&=&\int_{S}{\mathbb{E}[{{e}^{-a{{\zeta }^{(\sigma )}}}}|X_{\theta }^{(\sigma )}=y]p(x,dy)}\\
\label{1231203}&=&\int_{S}{{{\mathbb{E}}_{y}}{{e}^{-a{{\zeta }^{(\sigma )}}}}p(x,dy)}.
\end{eqnarray}
Fix a path $s\in\partial\mathbb{T}$ that contains vertex $\sigma$. Because of the time-homogeneity of the Markov chain $X_{s|0}$, $X_{s|1}$, $X_{s|2}$,\ldots one has
\[{{\mathbb{E}}_{\gamma }}\prod\limits_{j=0}^{n}{\frac{\lambda ({{X}_{s|j}})}{a+\lambda ({{X}_{s|j}})}}={{\mathbb{E}}_{\gamma }}\prod\limits_{j=0}^{n}{\frac{\lambda (X_{s|j}^{(\sigma )})}{a+\lambda (X_{s|j}^{(\sigma )})}}\ \ \ \forall n\in \mathbb{N}.\]
By \eqref{1231201},
\[{{2}^{n}}{{\mathbb{E}}_{\gamma }}\prod\limits_{j=0}^{n}{\frac{\lambda (X_{s|j}^{(\sigma )})}{a+\lambda (X_{s|j}^{(\sigma )})}}\le {{e}^{-n\delta }}\ \ \ \forall n\in \mathbb{N}.\]
One can apply the estimates in \eqref{1231202} with $\zeta$, $X_v$, $x$ being replaced by $\zeta^{(\sigma)}$, $X^{\sigma}_v$, $y$, respectively. Thus, $\mathbb{E}_ye^{-a\zeta^{(\sigma)}}=0$ for $\gamma$-a.e.\ $y\in S$. Because $p(x,dy)\ll\gamma(dy)$, one has $\mathbb{E}_ye^{-a\zeta^{(\sigma)}}=0$ for $p(x,\cdot)$-a.e.\ $y\in S$. Then \eqref{1231203} implies that $\mathbb{E}[e^{-a\zeta^{(\sigma)}}|X_\theta=x]=0$ for $\sigma\in\{1,2\}$. By \eqref{1231204}, $\mathbb{E}_xe^{-a\zeta}=0$. Therefore, $\zeta=\infty$ a.s.
\end{proof}

\begin{cor}\label{specrad}
Let $\{\lambda(X_v)^{-1}T_v\}_{v\in\mathbb{T}}$ be a DSY cascade of type ($\mathcal{M}$) with time-reversible probability measure $\gamma$. If the spectral radius of $T_a:L^2(\gamma)\to L^2(\gamma)$ or its operator norm is strictly less than $1/2$ for some $a>0$ then the conclusions in \autoref{noexplosion} holds.
\end{cor}

\begin{proof}
Denote by $\rho(T_a)$ the spectral radius of $T_a$. Because $\rho(T_a)\le \|T_a\|$, we can assume $\rho(T_a)<1/2$. By Cauchy-Schwarz inequality,
${{\left\langle 1,T_{a}^{n}1 \right\rangle }_{\gamma }}\le {{\left\| T_{a}^{n}1 \right\|}_{{{L}^{2}}(\gamma )}}\le \left\| T_{a}^{n} \right\|$. By Gelfand's formula,
\[\underset{n\to \infty }{\mathop{\lim \sup }}\,\sqrt[n]{{{\left\langle 1,T_{a}^{n}1 \right\rangle }_{\gamma }}}\le \underset{n\to \infty }{\mathop{\lim \sup }}\,\sqrt[n]{\left\| T_{a}^{n} \right\|}=\rho ({{T}_{a}}).\]
\end{proof}
\noindent In the next proposition, we give another sufficient condition, easier to verify, for a DSY cascade to be non-explosive. For this purpose, we strengthen the hypothesis by assuming:
\begin{itemize}
\item[(C)] \emph{There is a positive measure $m$ on $(S,{\mathcal S})$ such that
$p(x,dy)\ll m(dy)$ for every $x\in S$ and $\gamma(dx)\ll m(dx)$.}
\end{itemize}
Denote by $p(x,y)$ and $\gamma(x)$ the respective Radon-Nikodym
derivatives. We have the  {\it detailed balance} condition
 $$p(x,y)\gamma(x) = p(y,x)\gamma(y), \ \ m{\rm -a.e.}\ x,y\in S.$$
\begin{prop}\label{kppnonexp}
Let $\{\lambda(X_v)^{-1}T_v\}_{v\in\mathbb{T}}$ be a DSY cascade of type ($\mathcal{M}$) with condition (C). 
Assume further that the following trace condition holds
\begin{equation}
\label{extracond}
\int_Sg_1(x)^2p^{(2)}(x,x)m(dx) < \infty
\end{equation}
where $p^{(2)}$ is the two-step
transition \[p^{(2)}(x,y) = \int_Sp(x,z)p(z,y)m(dz), \ x,y\in S.\]
Then for $\gamma$-a.e.\ $x\in S$, the cascade is non-explosive for initial state $X_\theta=x$. If, in addition, $p(x_0,dy)\ll\gamma(dy)$ for some $x_0\in S$ then the cascade associated with the initial state $X_\theta=x_0$ is non-explosive. 
\end{prop}
\begin{remark}
A sufficient condition for \eqref{extracond} is
\begin{equation}\label{eq:23212}\int_{S}^{ }{p^{(2)}(x,x)m(dx)}<\infty.\end{equation}
\end{remark}
\begin{proof}
For $f\in L^2(\gamma)$, by Cauchy-Schwarz's inequality, 
\begin{eqnarray*}
|{{T}_{a}}f(x)|&=&g_a(x)\int_{S}^{ }{|f(y)|\sqrt{\gamma (y)}\frac{p(x,y)}{\sqrt{\gamma (y)}}m(dy)}\\
&\le& g_a(x){{\left\| f \right\|}_{{{L}^{2}}(\gamma )}}\sqrt{\int_{S}^{ }{\frac{p{{(x,y)}^{2}}}{\gamma (y)}}m(dy)}.
\end{eqnarray*}
Squaring and multiplying both sides by $\gamma(x)$, and using the detailed balance, we get
\begin{eqnarray*}
{{T}_{a}}f{{(x)}^{2}}\gamma (x)&\le&{g_a(x)^2}\left\| f \right\|_{{{L}^{2}}(\gamma )}^{2}\int_{S}^{ }{\frac{p{{(x,y)}^{2}}\gamma (x)}{\gamma (y)}}m(dy)\\
&=&{g_a(x)^2}\left\| f \right\|_{{{L}^{2}}(\gamma )}^{2}\int_{S}^{ }{p(x,y)p(y,x)}m(dy)\\
&=&{g_a(x)^2}\left\| f \right\|_{{{L}^{2}}(\gamma )}^{2}{{p}^{(2)}}(x,x).
\end{eqnarray*}
Integrating with respect to measure $m(dx)$ leads to
\[\left\| {{T}_{a}}f \right\|_{{{L}^{2}}(\gamma )}^{2}\le \left\| f \right\|_{{{L}^{2}}(\gamma )}^{2}\int_{S}^{ }{{{F}_{a}}(x)m(dx)}\]
where ${{F}_{a}}(y)={{g_a(x)^2}p^{(2)}(x,x)}$.
Thus, $\left\| {{T}_{a}} \right\|_{{{L}^{2}}(\gamma )\to {{L}^{2}}(\gamma )}^{2}\le {{\left\| {{F}_{a}} \right\|}_{{{L}^{1}(m)}}}$. Note that $\lim_{a\to\infty}F_a(x)=0$ for all $x>0$, $F_a(x)\le F_1(x)$ for all $a>1$, and that $F_1\in L^1(m)$. By Lebesgue's Dominated Convergence Theorem, $\|F_a\|_{L^1(m)}\to 0$ as $a\to\infty$. Therefore, there exists $a>0$ such that $\|T_a\|_{L^2(\gamma)\to L^2(\gamma)}<1/2$. The cascade is non-explosive according to \autoref{specrad}.
\end{proof}

\begin{cor}\label{besselnonexp}
Let $\{\lambda(X_v)^{-1}T_v\}_{v\in\mathbb{T}}$ be a DSY cascade of type ($\mathcal{M}$) with condition (C). Suppose that
\[\underset{x>0}{\mathop{\sup }}\,\lambda (x)^bp^{(2)}(x,x)<\infty ,\ \ \ \ \int_{S}^{ }{\frac{\lambda (x)^{2-b}}{{{(1+\lambda (x))}^{2}}}m(dx)}<\infty \]
for some $0\le b\le 2$. Then for $\gamma$-a.e.\ $x\in S$, the cascade is non-explosive for initial state $X_\theta=x$. If, in addition, $p(x_0,dy)\ll\gamma(dy)$ for some $x_0\in S$ then the cascade associated with the initial state $X_\theta=x_0$ is non-explosive. 
\end{cor}
\begin{proof}
It is easy to see that \eqref{extracond} is satisfied.
\end{proof}

\section{DSY cascades on non-binary trees}\label{non-binary}

Although our interest is mainly on DSY cascades on a binary tree, which are well-suited with PDEs with quadratic nonlinearity, the techniques we used above can be applied to trees with random numbers of offspring, for example, Galton-Watson trees. 
Namely, let $\cV=\{\theta\}\cup \bigcup_{n\in\mathbb{N}}\mathbb{N}^n$ be the set of all possible vertices with $\theta$, as usual, denoting the root. 
Let $\{\lambda_v\}_{v\in \cV}$ be a family of positive random variables representing the intensities and $\{T_v\}_{v\in \cV}$ be a family of i.i.d.\ mean-one exponential random variables. Let $\cT\subset \cV$ be a random subtree of $\cV$, rooted at $\theta$.

\begin{definition}
Suppose the random structures $\cT$, $\{\lambda_v\}_{v\in \cV}$, and  $\{T_v\}_{v\in \cV}$ are independent. Then, we refer to the triplet $(\cT$, $\{\lambda_v\}_{v\in \cV},  \{T_v\}_{v\in \cV})$ as a {\em doubly stochastic Yule (DSY) cascade on a random tree structure} $\cT$. In analogy with the binary DSY cascades, we will use the notation $\{{\lambda^{-1}_vT_v}\}_{v\in{\cT}}$ for DSY cascades on random trees.
\end{definition} 

The essence of explosion of a DSY cascade is the occurrence of infinitely many
exponential clock ``rings’’ within a finite time horizon.  In particular,  finite
trees should be non-explosive. On the other hand, a general random tree structure may contain both finite (terminating) paths and infinite paths. A reasonable definition of explosion times is one in which any finite path has an infinite ``length''. A natural way to capture this feature is to assign the waiting time between a terminal vertex (leaf) and the next branching to
be infinite. We thus arrive at the following definition of the explosion time:
\begin{equation}
\label{explosiontime}
\zeta=\sup_{n\ge 0}\,\inf\limits_{|v|=n,\, v\in \cV}\sum_{j=0}^n\frac{T_{v|j}}{\lambda_{v|j}}\,(\IND{v|j\in \cT})^{-1},
\end{equation}
with the convention that $\frac{1}{0}=\infty$. Note that in the case of a binary tree this definition of $\zeta$ is consistent with \autoref{expl_time}.
As before, we refer to the event $\zeta<\infty$ as the {\em explosion event}. This notion of explosion is consistent with the intuitive idea illustrated in \autoref{cascade} (an analog of \autoref{infinitepath}): if $\zeta<t<\infty$ then there exists an infinite random path (the shortest path) of the DSY cascade that does not reach time $t$, and thus the tree has generated infinitely many vertices by that time. In contrast, observe that if the tree $\cT$ is subcritical (i.e. has a finite number of vertices), then $\zeta=\infty$  and the DSY cascade is automatically non-explosive. This is the case of Galton-Watson tree with the mean number of offspring $\mu\le1$ and the case of the thinned DSY-type cascade constructed by Le Jan and Sznitman for the Navier-Stokes equations \cite{YLJ_AS_1997}.
 
The key lemma (\autoref{nonexpbound}) can be extended to the case of trees with the random number of offspring as follows.
\begin{lemma}
\label{nonexpbound-gw}
Let $\{\lambda_v^{-1}T_v\}_{v\in\mathcal{T}}$ be a DSY cascade on a random tree structure $\cT$. Assume that, almost surely,  each vertex of $\cT$ has at least one offspring in $\cT$ and has mean number of offspring bounded by $\mu<\infty$. 
Suppose further that for each $s\in \partial\cV:=\mathbb{N}^{\infty}$, the distribution of the sequence $\lambda_{s|0}$, $\lambda_{s|1}$, $\lambda_{s|2}$, $\ldots$ does not depend on $s$.
Then for $a>0$ and an arbitrary fixed path ${s}\in\partial{\cV}$,
\begin{equation}
\label{noexpbound-gw}
{\mathbb E} e^{-a\zeta}\le \liminf_{n\to\infty}\mu^n{\mathbb E}\prod_{j=0}^n{\frac{\lambda_{{s}|j}}{
a+\lambda_{{s}|j}}}.
\end{equation}
Consequently, if 
\begin{equation}\label{noexpbound1-gw}
\liminf_{n\to\infty}\mu^n{\mathbb E}\prod_{j=0}^n{\frac{\lambda_{{s}|j}}{
a+\lambda_{{s}|j}}}=0
\end{equation} 
for some $a>0$ then the cascade is non-explosive. 
\end{lemma}
\begin{proof}
Let $V_n=\#\{v\in\cT:\ |v|=n\}$ be the random number of vertices in $\cT$ of generation $n$. First, note that ${\mathbb E}V_n \le \mu^n, n\ge 0$. By conditioning on $V_n$ (Wald's identity \cite{RB_EW_2017}),
\begin{eqnarray*}
\label{genmaxbysum}
{\mathbb E} e^{-a\zeta} &\le& 
\liminf_{n\to\infty}{\mathbb E}\exp\left({-\min_{|v|=n,\ v\in \cV}\sum_{j=0}^n a \frac{T_{v|j}}{\lambda_{v|j}}}\,(\IND{v|j\in \cT})^{-1}\right)\nonumber\\
&\le& \liminf_{n\to\infty}{\mathbb E}\sum_{|v|=n,\, v\in \mathcal{T}}\exp\left({-\sum_{j=0}^na \frac{T_{v|j}}{\lambda_{v|j}}}\right)\nonumber\\
&=& \liminf_{n\to\infty}{\mathbb E}V_n\,{\mathbb E}\exp\left({-\sum_{j=0}^na \frac{T_{{s}|j}}{\lambda_{{s}|j}}}\right)\nonumber\\
&\le& \liminf_{n\to\infty}\mu^n{\mathbb E}\prod_{j=0}^n\frac{\lambda_{{s}|j}}{
a+\lambda_{{s}|j}}.
\end{eqnarray*}
\end{proof}

\begin{remark}
Thanks to \autoref{nonexpbound-gw}, \autoref{noexplosion} and its corollaries extend naturally to DSY cascades on trees with random number of branches.
\end{remark}

\section{Examples}\label{examples}

The following example includes a large
class of DSY with time-reversible Markov process intensities and helps to clarify the role of the
additional trace condition in \autoref{kppnonexp}.

\begin{example}[Birth-Death Intensities]
Consider a type $(\mathcal{M})$ DSY with $\lambda(x)=x$ and a family of $\mathbb{N}$-valued random variables $\{X_v\}_{v\in\mathbb{T}}$ 
distributed with transition probabilities $p_{j,k}$ where 
\[p_{j,j+1}=\mathbb{P}(X_{v*1}=j+1\,|\,X_v=j)=\mathbb{P}(X_{v*2}=j+1\,|\,X_v=j)=\beta_j,\]
\[p_{j,j-1}=\mathbb{P}(X_{v*1}=j-1\,|\,X_v=j)=\mathbb{P}(X_{v*2}=j-1\,|\,X_v=j)=\delta_j,\]
where $\beta_1=1$, and $\delta_j = 1-\beta_j\in (0,1)$ for $j=2,3,\dots$ Here $v*1$ and $v*2$ denote the two offspring of vertex $v$. Along each path $s\in\partial\mathbb{T}$, the sequence $X_{s|0}$, $X_{s|1}$, $X_{s|2}$, \ldots is the birth-death process on the state space $S=\mathbb{N}$
 with reflection at $1$ and
birth-death rates $\beta_j,\,  \delta_j$ (see \cite{RB_EW2009},
p. 238-246).  This is an ergodic time-reversible
Markov process (see \cite{RB_EW2009}, Theorem 3.1(b), p. 241) with invariant
probability
\begin{equation}
\gamma_j = \frac{\beta_2\cdots\beta_{j-1}}{\delta_2\dots\delta_j}\gamma_1,\ \ \ j=2,3,\dots,
\end{equation}
provided that
$$\gamma_1=\sum_{j=2}^\infty\frac{\beta_2\cdots\beta_{j-1}}{\delta_2\dots\delta_j}<\infty.$$
Also
\begin{equation}
p^{(2)}_{j,j} = p_{j,j-1}p_{j-1,j} + p_{j,j+1}p_{j+1,j}
= (1-\beta_j)\beta_{j-1} + \beta_j(1-\beta_{j+1}).
\end{equation}
The trace condition \eqref{eq:23212} becomes
 $\sum_{j=1}^\infty p^{(2)}_{j,j} < \infty$. This condition together with the finiteness of  $\gamma_1$
  implies $\beta_j\to 0$ as
 $j\to\infty$, i.e., a stronger tendency to return to smaller states
 from states far away, which is a stronger condition than the ergodicity alone.
 \end{example}

\begin{example}[Bessel Cascade for NSE]\label{NSE_ex}

The Bessel cascade of the Navier-Stokes equations is a DSY cascade of type ($\mathcal{M}$) with $\lambda(x)=x^2$ and $\{X_v=|W_v|\}_{v\in\mathbb{T}}$ (the wave number magnitudes) \cites{RD_NM_ET_EW_2015, YLJ_AS_1997}. The transition probabilities have a density
\[p(x,y)=\begin{dcases}{\frac{e^{2x}-1}{x}}e^{-2y} &\mbox{if } x<y\\
\frac{1-e^{-2y}}{x} &\mbox{if } x\ge y.
\end{dcases}\]
One can check that along each path $s\in\partial\mathbb{T}$, the Markov process $X_{s|0}$, $X_{s|1}$, $X_{s|2}$, \ldots is time reversible with respect to the unique invariant probability density $\gamma(x)=4xe^{-2x}$, $x>0$. These transition probabilities are also realized by
the iterated maps 
\begin{eqnarray*}
X_{v*1} = U_{{v}*1}X_v + \frac{1}{2}T_{{v}*1},\ \ \ 
X_{v*2} = U_{{v}*2}X_v + {\frac{1}{2}}T_{{v}*2},
\end{eqnarray*}
where $(U_1,U_2), (U_{11},U_{12}), (U_{21},U_{22}),\dots$ is an i.i.d.\ family of bivariate random vectors uniformly distributed on the 
diagonal of the square $(0,1)\times(0,1)$, i.e., $U_1$ and $U_2$ are
each uniform on $(0,1)$ and $U_1+U_2=1$, and $\{T_v\}_{v\in\mathbb{T}}$ is a family of i.i.d.\ mean one exponentially distributed
random variables, independent of the $U$'s. 
In view of its mean-reversion character to unity, and
 the non-explosive character of the standard Yule process \cite{WF1968}, 
one might guess that the Bessel cascade is non-explosive.\footnote{This informal thinking lead to a previous erroneous proof in 
\cite{RD_NM_ET_EW_2015}*{Prop.\ 5.1 in the Appendix}, although the assertion remains valid as shown
in the present paper.} We will use \autoref{besselnonexp} (with $b=1$) to show that the Bessel cascade is non-explosive.
\[\int_{0}^{\infty }{p(x,y)p(y,x)}dy=\underbrace{\int_{0}^{x}{\frac{1-{{e}^{-2y}}}{x}\frac{{{e}^{2y}}-1}{y}{{e}^{-2x}}dy}}_{\{1\}}+\underbrace{\int_{x}^{\infty }{\frac{{{e}^{2x}}-1}{x}{{e}^{-2y}}\frac{1-{{e}^{-2x}}}{y}dy}}_{\{2\}}.\]
It suffices to show that $x^2\{1\}$ and $x^2\{2\}$ are bounded functions on $(0,\infty)$. We have
\[{{x}^{2}}\{1\}=\frac{\int_{0}^{x}{{{{({{e}^{y}}-{{e}^{-y}})}^{2}}}/{y}\,dy}}{{{e}^{2x}}/x}.\]
By L'Hospital Rule,
\[\underset{x\to \infty }{\mathop{\lim }}\,{{x}^{2}}\{1\}=\frac{{{({{e}^{x}}-{{e}^{-x}})}^{2}}/x}{(2x-1){{{e}^{2x}}}/{{{x}^{2}}}}=\frac{1}{2}.\]
Thus, $x^2\{1\}$ is a bounded function on $(0,\infty)$. On the other hand,
\[\{2\}\le \int_{x}^{\infty }{\frac{{{e}^{2x}}-1}{x}{{e}^{-2y}}\frac{1-{{e}^{-2x}}}{x}dy}=\frac{({{e}^{2x}}-1)(1-{{e}^{-2x}})}{{{x}^{2}}}\int_{x}^{\infty }{{{e}^{-2y}}dy}<\frac{1}{2x^2}.\]
This concludes the proof of the non-explosion of the Bessel cascade for every initial state $X_\theta=x>0$.
\end{example}

\begin{example}[A Mean-Field Cascade]
Let $\{X_v\}_{v\in\mathbb{T}}$ be a family a random variables such that along each path $s\in\partial\mathbb{T}$ the sequence $X_{s|1}$, $X_{s|2}$, $X_{s|3}$,\ldots is an i.i.d.\ sequence of random variables with distribution $\gamma(dx)$. For any positive measurable function $\lambda$ defined on the state space, one can check that $\{\lambda(X_v)^{-1}T_v\}_{v\in\mathbb{T}}$ is a DSY cascade of type ($\mathcal{M}$). The Markov chain along each path has transition probabilities $p(x,dy)=\gamma(dy)$. 
For $a>0$ and $s\in\partial\mathbb{T}$,
\[{{2}^{n}}\mathbb{E}\prod\limits_{j=0}^{n}{\frac{\lambda ({{X}_{s|j}})}{a+\lambda ({{X}_{s|j}})}}\le {{2}^{n}}\mathbb{E}\prod\limits_{j=1}^{n}{\frac{\lambda ({{X}_{s|j}})}{a+\lambda ({{X}_{s|j}})}}={{(2\mathbb{E}{{Y}_{a}})}^{n}},\]
where $Y_a=\lambda(X_1)/(a+\lambda(X_1))$. Note that $\lim_{a\to\infty}\mathbb{E}Y_a=0$ by Lebesgue's Dominated Convergence Theorem. Therefore, for sufficiently large $a>0$,
\[\underset{n\to \infty }{\mathop{\lim \inf }}\,{{2}^{n}}\mathbb{E}\prod\limits_{j=0}^{n}{\frac{\lambda ({{X}_{s|j}})}{a+\lambda ({{X}_{s|j}})}}=0.\]
By \autoref{nonexpbound}, the cascade is non-explosive (for any initial distribution).
\end{example}

\begin{example}[Cascade for KPP equation]\label{KPP_ex} The well-known KPP equation (in the \emph{physical space}) 
has yielded highly successful
 theories for branching Brownian motion and branching random walk as
documented, for example, in \cites{MB_1983, AK_1998}. In the \emph{Fourier space}, the equation is associated with a DSY cascade as detailed below. We will apply \autoref{kppnonexp} to show the non-explosion of the cascade. 
The same cascade was analyzed by Orum \cite[Sec.\ 7.9]{orum}, where the non-explosion was established by a different method (via the uniqueness of solutions to the equation).
Recall the KPP equation 
\begin{equation}
\label{KPPeqn}
\frac{\partial u}{\partial t} = \frac{\partial^2u}{\partial x^2} + u^2-u, \quad u(x,0) = u_0(x), x\in{\mathbb R},
\end{equation}
where we have omitted the typical coefficient $1/2$
of the Laplacian as a matter of notational convenience
on the Fourier side. The cascade model of this equation in the Fourier space is a discrete parameter branching 
Markov chain
obtained as follows. Taking Fourier transforms and 
 expressing \eqref{KPPeqn} in integrated form, one arrives at
\begin{equation}
\label{FKPPeqn}
\hat{u}(\xi,t) = \hat{u}_0(\xi)e^{-(1+\xi^2)t} + 
\int_0^t\int_{\mathbb R}e^{-(1+\xi^2)s}
\hat{u}(\eta,t-s)\hat{u}(\xi-\eta,t-s)d\eta ds.
\end{equation}
Here $\hat{f}(\xi) = \frac{1}{\sqrt{2\pi}}\int_{-\infty}^\infty e^{-i\xi x}f(x)dx, \xi\in{\mathbb R}$ denotes the Fourier transform
of an integrable function $f$. 
Defining $\chi(\xi,t) = \frac{\hat{u}(\xi,t)}{h(\xi)}$, for a 
positive function $h$ to be determined, one has
\begin{equation*}
\label{hFKPP}
\chi(\xi,t) = {\chi}_0(\xi)e^{-(1+ \xi^2)t} + 
\int_0^t\int_{\mathbb R}(1+\xi^2)e^{-(1+\xi^2)s}
\chi(\eta,t-s)\chi(\xi-\eta,t-s) \frac{h(\eta)h(\xi-\eta)}{(1+\xi^2)h(\xi)}d\eta ds.
\end{equation*}
The positive function $h$, referred to as a \emph{majorizing kernel} \cite{RBetal_2003}, is determined such that
\begin{equation}
H(\eta|\xi) = \frac{h(\eta)h(\xi-\eta)}{(1+\xi^2)h(\xi)}
\end{equation}
is a probability kernel. Thus, $h$ is a positive
function satisfying
\begin{equation}
h*h(\xi) = (1+\xi^2)h(\xi),\quad \xi\in{\mathbb R}.
\end{equation}
An analysis of this equation yields\footnote{The hyperbolic
cosecant distribution belongs to the family of so-called
generalized hyperbolic secant distributions, and has a
relatively rich history in mathematical statistics originating
with R. Fisher; see \cite{LD_1993}.} a solution
 $h(\xi) = 3\xi \text{csch}(\pi\xi),\xi\in{\mathbb R}$; see \cite[p.\ 146]{orum}.  This majorizing
 kernel determines an
ergodic Markov process $W_{s|0}=\xi, W_{s|1},\,W_{s|2},\,\dots$
along a path $s\in\partial\mathbb{T}$ with transition 
probabilities $H(\eta|\xi)d\eta$. This Markov process is
time-reversible with respect to the unique invariant distribution
$\gamma(d\xi) = (1+\xi^2)h^2(\xi)d\xi, \xi\in{\mathbb R}$.  

The cascade associated with the KPP equation is $\{\lambda(X_v)^{-1}T_v\}_{v\in\mathbb{T}}$ where $\{X_v=W_v\}$ and $\lambda(\xi)=1+\xi^2$. This is a DSY cascade of type ($\mathcal{M}$) with transitional distribution $p(\xi,\eta)d\eta=H(\eta|\xi)d\eta$
along each path. The Markov process is time reversible with respect to the probability measure 
\[\gamma(d\xi)=\frac{5\pi}{9}(1+\xi^2)h(\xi)^2d\xi.\]
In this case, $p(\eta,d\xi)\ll \gamma(d\xi)\ll m(d\xi)$ for all $\eta\in\mathbb{R}$, where $m$ is the Lebesgue measure. Because $0<h(\xi)<2$ for all $\xi$, we have
\begin{eqnarray*}\int_{0}^{\infty }{\int_{0}^{\infty }{p(\xi,\eta)p(\eta,\xi)d\eta d\xi}}&=&\int_{0}^{\infty }{\int_{0}^{\infty }{\frac{{h(\xi-\eta)^2}}{(1+{{\xi}^{2}})(1+{{\eta}^{2}})}d\eta d\xi}}\\
&<&\int_{0}^{\infty }{\int_{0}^{\infty }{\frac{4}{(1+{{\xi}^{2}})(1+{{\eta}^{2}})}d\eta d\xi}}\\
&=&{{\left( \int_{0}^{\infty }{\frac{2}{1+{{\xi}^{2}}}d\xi} \right)}^{2}}<\infty .
\end{eqnarray*}
By \autoref{kppnonexp}, the cascade is non-explosive for every initial state $X_\theta=\xi\in\mathbb{R}$.

\end{example}

\section{Closing Remarks}\label{conclusion}
The non-explosion 
criteria provided by the main theorem apply
to natural stochastic problems arising in the analysis of a class of important
nonlinear PDEs.   The models
  may also be viewed in the context  as generalization of a branching model arising
  in computer science, statistical physics,  and cellular biology.

 To dispense with the time-reversibility condition obviously requires
  a completely different approach than that involving self-adjoint operators
  on $L^2$.  The authors
  introduce a probabilistic
   ``cutset method'' in \cite{DTTW_22} to obtain further sufficient conditions
   for non-explosion
  in the absence of the time-reversibility assumption.  In addition, 
  criteria for explosion which are applicable to the 
  Navier-Stokes equations and more purely probability models are 
  also developed in \cite{DTTW_22}.
   An analytic proof by PDEs has also been obtained
in \cite{DTT_22} for the Bessel cascade example.

\section*{Acknowledgment}
The authors would like to thank the referees for their careful reading, comments, and helpful suggestions that led to an improvement of the exposition.
\appendix

\section{Appendix: DSY Cascades Arising from Evolutionary PDEs}\label{PDE-casc}

Our analysis of the
 non-explosion/explosion problem for DSY cascades is primarily motivated by its connection with semilinear evolutionary PDEs with quadratic nonlinearity. If the nonlinear term is a simple product not involving derivatives, e.g.\ in the case of the Fisher-KPP equation, the relation between the existence and uniqueness and the
 branching diffusions is well-established since the early work of 
 Ikeda, Nagasawa, and Watanabe \cites{ikeda1968branching1, ikeda1968branching2, ikeda1969branching3}, 
and that of It\^o and McKean \cite[pp.\ 206-211]{ito2012diffusion}.  Briefly speaking, this involves a probabilistic
representation of a solution $u(t,x)$ to a scalar evolutionary PDE
whose linear term is the infinitesimal generator $A$ of a diffusion and whose
 nonlinear term is of the form $\sum_jp_ju^j-u$, where $\sum_jp_j=1$,
$p_j\ge 0$. The case 
$A=\Delta$ (the infinitesimal generator of Brownian
motion) and $p_2=1$ corresponds to the classical 
Fisher-KPP  equation. 
This type of branching process in physical space may be used to establish both global-in-time existence 
as well as finite-time blowup results for solutions to the aforementioned equations under suitable conditions \cites{Fujita, Lopez, Nagasawa}. 

However, 
vectorial evolutionary PDEs involving 
{\it derivatives} in the nonlinear term, 
such as the Navier-Stokes equations, are outside of the scope of that theory.
The presence of derivatives in the nonlinear term naturally leads to the introduction of
 Fourier scales in the underlying stochastic cascade. More specifically, the waiting times are dependent on Fourier wave-vectors. This is an important feature 
 distinguishing classical Yule cascades from the DSY cascades considered in this paper. Thus, in order to identify a stochastic structure 
 intrinsic to the Navier-Stokes equations, 
 it is natural to consider the equations in the Fourier space \cite{YLJ_AS_1997}. In this setting, derivatives become Fourier multipliers, and one obtains a
 mild-type formulation of the equation equivalent to an averaging of the underlying stochastic cascade structure. 
It is noteworthy that 
 the stochastic cascade representation of solutions in the Fourier space  provides a unified framework that applies to the
DSY cascades of {\em general}  semilinear evolutionary PDEs including the  Fisher-KPP equation (\autoref{KPP_ex}).

A common feature of most evolutionary equations that generate a DSY cascade is that they define a dissipative dynamical system which, when formulated in the Fourier space, has a linear term that determines the intensities of the exponential waiting times
between branchings,  and a quadratic nonlinear term that 
yields a random binary tree.
  In the examples in \autoref{examples}, the equations can be written in the Fourier space as
\begin{equation}
\label{generic}
\hat{u}(\xi,t) = e^{-\lambda(\xi)t} \hat{u}_0(\xi) + \int_0^t e^{-\lambda(\xi)s} \rho(\xi) \int_{\mathbb{R}^d} B_{\xi}(\hat{u}(\eta,t-s), \hat{u}(\xi-\eta,t-s))\, d \eta ds
\end{equation}
where $\lambda$, $\rho$ are radially symmetric positive functions, and $B_{\xi}(\cdot,\cdot)$ is a bilinear map. The functions $\lambda$, $\rho$, $B_\xi$ are determined by the specific PDE under consideration.  
For example, in the case of the incompressible Navier-Stokes equations in $\mathbb{R}^d$ \cite{YLJ_AS_1997, RD_NM_ET_EW_2015}:
\[
\lambda(\xi)=\nu|\xi|^2,\quad \rho(\xi)=|\xi|,\quad \mbox{and}\quad
B_{\xi}(\hat{u}(\eta,t-s), \hat{u}(\xi-\eta,t-s))=\hat{u}(\eta,t-s)\odot_{\xi}\hat{u}(\xi-\eta,t-s),
\]
with
\[
v\odot_{\xi} w=-i (v\cdot e_{\xi})\, \mathrm{Pr}_{{\xi^{\perp}}}w,\]
where $e_\xi={\xi}/{|\xi|}$, and $\mathrm{Pr}_{{\xi^{\perp}}}w$ is the orthogonal projection of $w$ on the plane orthogonal to $\xi$. The presence of projections is due to the Leray projection of the nonlinear term in Fourier space.

A key step in the probabilistic reformulation of \eqref{generic} is to find a function 
$h(\xi)$ such that 
\begin{equation}\label{functionH}
H(\eta|\xi) = \frac{\rho(\xi)}{\lambda(\xi)}\frac{h(\eta) h(\xi-\eta)}{h(\xi)}
\end{equation}
is a probability density function on $\mathbb{R}^d$.  Once $h$ is identified, we introduce a new unknown $\chi(\xi,t) = \hat{u}(\xi,t)/h(\xi)$, which satisfies the normalized equation
\begin{equation}
\label{genericfourier}
{\chi}(\xi,t) = e^{-\lambda(\xi)t} \chi_0(\xi) + \int_0^t e^{-\lambda(\xi)s} \lambda(\xi) \int_{\mathbb{R}^d} B_{\xi}(\chi(\eta,t-s), \chi(\xi-\eta,t-s)) H(\eta|\xi)\; d\eta ds.
\end{equation}

The solution $\chi(\xi,t)$ to \eqref{genericfourier} can be expressed as the expected value
of  ``{\em solution}"  
stochastic process $\bX(\xi,t)$ satisfying (in distribution):
\begin{equation}\label{gen-recursion}
\bX(\xi ,t)=\left\{ \begin{array}{*{35}{r}}
   {{\chi }_{0}}(\xi ) & \text{if} & {{T}_{\theta }}/\lambda(\xi)\ge t  \\
   B_{\xi}\left({\bX^{(1)}}({{W}_{1}},t-{{T}_{\theta }}),{\bX^{(2)}}({{W}_{2}},t-{{T}_{\theta }})\right) & \text{if} & {{T}_{\theta }}/\lambda(\xi)< t  \\
\end{array} \right.
\end{equation}
where $T_\theta\sim\text{Exp}(1)$, $W_1\sim H(\cdot|\xi)$, $W_2=\xi-W_1$, and $\bX^{(1)}$ and $\bX^{(2)}$ are independent copies of $\bX$. Here, the symbol $\sim$ is used to convey the distribution.

The recursion \eqref{gen-recursion} leads to a family of random wave vectors $\{W_v\}_{v\in\mathbb{T}}$ satisfying $W_\theta=\xi$, $W_{v*1}+W_{v*2} = W_v$ for all $v\in\mathbb{T}$ and, conditionally given $W_v$, $W_{v*1}$ and $W_{v*2}$ are each distributed as $H(\cdot|W_v)$. For $X_v=W_v$, one gets a DSY cascade $\{\lambda(X_v)^{-1} T_v\}_{v\in\mathbb{T}}$ according to \autoref{DSYMcascade}. In most cases, the waiting times between 
branchings only depends on 
the magnitudes of the  random wave vectors which, in turn,  
have a well-behaved branching Markov structure. For the incompressible Navier-Stokes equations in $\mathbb{R}^3$, the choice of  $X_v=|W_v|$  turns out to be more efficient than the choice of $X_v=W_v$ \cites{YLJ_AS_1997, RD_NM_ET_EW_2015}. 

In the case $t<\zeta$ (where $\zeta=\zeta(\xi)$ is the explosion time, see \autoref{expl_time}), sample realizations of
$\bX(\xi,t)$ are uniquely defined 
by \eqref{gen-recursion}  as an iterated composition of 
$B(\cdot,\cdot)$ with the initial data  evaluated at the leaves along the corresponding DSY cascade  (see \autoref{NSE-casc}). 
In the case $t\ge\zeta$, there may be multiple solutions of \eqref{gen-recursion}, including the minimal solution process defined by setting $\bX(\xi,t)=0$ in the event $[t\ge\zeta]$ (see \cite{alphariccati,DTT_22}).
In particular, when $\zeta=\infty$, i.e.\ the case of nonexplosion, the solution process $\bX(\xi,t)$ is uniquely defined for all $t\ge 0$.
\begin{center}
\begin{figure}
  \centerline{\includegraphics[scale=1.1]{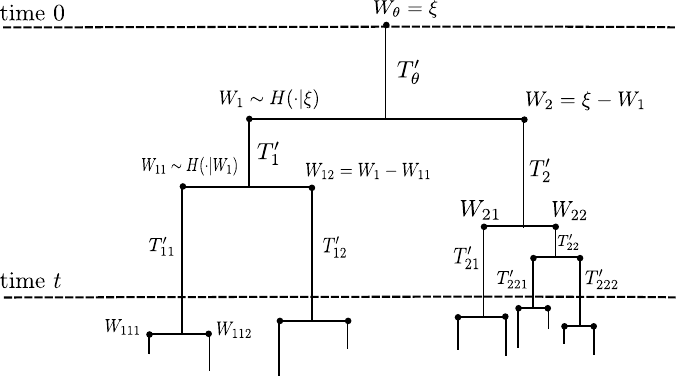}}
\caption{A realization of a non-exploding DSY for \eqref{genericfourier} where $T'_v=T_v/\lambda(W_v)$. In case of NSE, $B_{\xi}(v,w)=v\odot_{\xi}w$ and the solution process becomes:}
\begin{center}$\bX(\xi,t)= (\chi_0(W_{11})\odot_{_{W_{1}}}\chi_0(W_{12}))\odot_\xi [\chi_0(W_{21})\odot_{_{W_{2}}}(\chi_0(W_{221})\odot_{_{W_{22}}}  \chi_0(W_{222}))]$.
\end{center}
\label{NSE-casc}
\end{figure}
\end{center}

Thus, the stochastic explosion or non-explosion of the associated DSY cascades has interesting implications for the existence and uniqueness of global-in-time solutions of these equations \cite{RD_NM_ET_EW_2015, alphariccati, DTT_22}. For example, in the case of the $\alpha$-Riccati equation and the Montgomery-Smith equation \cite{montgomery2001finite}, the explosion of the underlying DSY cascades is used to show non-uniqueness of the initial value problems \cites{alphariccati, DTT_22}. This method also applies for the explosive DSY cascades associated with the generalized KPP equations in physical space in \cite[pp.\ 206-211]{ito2012diffusion}. Regarding the global-in time existence of solutions, the associated DSY cascades provide a pathway to establish global solutions for small initial data (in appropriate settings), consistent with the results obtainable by analytical techniques in the literature \cites{YLJ_AS_1997, lemarie2002, DTT_22, lemarie2016}. However, it is worth emphasizing that in the case of $\alpha$-Riccati equations, $0\le\alpha<1$,  and complex Burgers equation, the global existence of the solutions with arbitrarily large initial data can be proved directly from the corresponding DSY cascades, while in the case of $\alpha$-Riccati equations with $\alpha\ge1$ and the Montgomery-Smith equation, the finite-time blowup and uniqueness/non-uniqueness of the solutions can be established for solutions built on both explosive and non-explosive DSY cascades \cites{alphariccati, DTT_22, RD_NM_ET_EW2019}.

The most natural function space corresponding to cascade 
solutions is the Besov-type space determined by the scaling function $h$:
\[
\mathcal{F}_h=\{u:\ \|u\|_h=\mathrm{esssup}|\hat{u}/h|<\infty\}.
\]
However, other adapted spaces including weighted $L^p$ spaces may
also be considered (see  \cites{RBetal_2003, DTT_22}).
For the NSE in $\mathbb{R}^3$, there are 
two functions $h$ that make $H$ in \eqref{functionH} a probability kernel, 
both first obtained in \cite{YLJ_AS_1997}. 

One is $h_d(\xi)=c/|\xi|^2$, which yields a scale-invariant  (with respect to the natural scaling) probability kernel $H_d(\eta|\xi)=\frac{c|\xi|}{|\eta|^2|\xi-\eta|^2}$. Interested reader can refer to \cite{RD_NM_ET_EW_2015} for more detailed discussion of the connection between this kernel and self-similar solutions to NSE, and \cite{DTTW_22} for the explosion character of associated cascade. The function space 
$\mathcal{F}_{h_d}$ associated with this kernel is a scale-critical space. For the existence and uniqueness results of the 
cascade solutions $\hat{u}(\xi,t)=h(\xi) \mathbb{E}\mathbf{X}(\xi,t)$ in this space, see 
\cites{DTT_22, RBetal_2003, cannone2000, YLJ_AS_1997}, \cite[Sec.\ 8.7]{lemarie2016}.
 
The second scaling function is $h_b(\xi)=c {e^{-|\xi|}}/{|\xi|}$, which was found in \cite{YLJ_AS_1997} and generalized in \cite{RBetal_2003}. It is of the same type as the Bessel kernels  introduced in \cite{aronszajn1961}. In this spirit, we 
refer to the corresponding DSY cascade as the {\em Bessel cascade}. In contrast to the scale-invariant function $h_d$ mentioned earlier, the scaling function $h_b$ defines a smooth function space $\mathcal{F}_{h_b}$. Given the uniqueness of smooth solutions, we can expect that the underlying stochastic cascade should be non-explosive, which is shown in \autoref{NSE_ex}. 

While the well-posedness results for PDEs obtained using DSY cascades are consistent with the results obtained by traditional analytic approaches, improved understanding of these stochastic structures can provide a new mathematical framework
for  the existing theory and open up new avenues to study open questions in the qualitative theory of these PDEs.  In particular, solutions to NSE constructed from the non-explosive Bessel cascade belong to the Leray-Hopf class of weak solutions \cite{YLJ_AS_1997}, providing an additional method
 to view  the regularity problem of NSE.   On the other hand,
 since
   the self-similar cascade is explosive \cite{DTTW_22},  an entirely new framework to explore non-uniqueness (and possibly blow-up) of the mild-type solutions is made
   available, thus potentially
  complementing existing non-uniqueness and blow-up theory of weak solutions \cites{buckmaster2019, bourgain2008, albritton2021non}.

\bibliographystyle{unsrt}
\bibliography{DY_bib}

 \end{document}